\documentclass[12pt]{amsart}

\usepackage{fullpage}
\usepackage{amsmath}
\usepackage{amssymb}
\usepackage{amsthm}
\usepackage{amstext}
\usepackage{appendix}
\usepackage{perpage}
\usepackage{color}
\usepackage{xcolor}
\usepackage{tikz-cd}
\usepackage[colorlinks=true,linkcolor=blue,citecolor=blue]{hyperref}

\newtheorem{thm}{Theorem}[section]
\newtheorem{lemm}[thm]{Lemma}

\newtheorem{prop}[thm]{Proposition}
\newtheorem{conj}[thm]{Conjecture}

\theoremstyle{definition}
\newtheorem{expl}[thm]{Example}
\newtheorem{defi}[thm]{Definition}
\newtheorem{remark}[thm]{Remark}

\begin{document}

\title{Generic character sheaves on reductive groups over a finite ring}

\author{Zhe Chen}

\address{Department of Mathematics, Shantou University, Shantou, China}

\email{zhechencz@gmail.com}

\begin{abstract}
In this paper we propose a construction of generic character sheaves on reductive groups over finite local rings at even levels, whose characteristic functions are higher Deligne--Lusztig characters when the parameters are generic. We formulate a conjecture on the simple perversity of these complexes, and we prove it in the level two case (thus generalised a result of Lusztig from the function field case). We then discuss the induction and restriction functors, as well as the Frobenius reciprocity, based on the perversity.
\end{abstract}

\maketitle

\tableofcontents

\section{Introduction}

The theory of character sheaves, initiated by Lusztig in \cite{Lusztig_CharSh_I} for connected reductive groups over an algebraically closed field, is a geometric theory of characters. In this paper we propose a construction of generic character sheaves for connected reductive groups over $\mathcal{O}^{\mathrm{ur}}/\pi^r$, where $\mathcal{O}$ is a complete discrete valuation ring with a finite residue field $\mathbb{F}_q$, $\mathcal{O}^{\mathrm{ur}}$ its maximal unramified extension, $\pi$ a uniformiser, and $r$ a fixed arbitrary positive even integer. In the prior work \cite{Lusz_CharSh_Generalizations}, Lusztig considered the principal series case with $\mathcal{O}=\mathbb{F}_q[[\pi]]$ (but without any restriction on $r$); his construction is based on the use of a natural analogue of the Grothendieck--Springer resolution. In this paper, we will replace the resolution morphism by another morphism, based on the \emph{arithmetic radical} introduced in \cite{ChenStasinski_2016_algebraisation} (see Section~\ref{Deligne--Lusztig theory}); this replacement allows us to deal with any series, not only the principal series, and as one shall see, some methods in the $r=1$ case are also available in our construction. In the below we describe the motivation of our construction.

\vspace{2mm} In 1979, Lusztig \cite{Lusztig1979SomeRemarks} proposed a geometric method to study the ordinary representations of reductive groups over $\mathcal{O}/\pi^r$, which generalises Deligne and Lusztig's seminal work \cite{DL1976} (which corresponds to the case $r=1$). The proofs of some results in \cite{Lusztig1979SomeRemarks} were later established in the function field case by Lusztig himself in \cite{Lusztig2004RepsFinRings}, and Stasinski \cite{Sta2009Unramified} generalised this work from the function field case to the general case, by the use of Greenberg functor technique. Meanwhile, when $r\geq 2$, there also exists an algebraic method to construct certain irreducible representations of these groups, due to G\'{e}rardin \cite{Gerardin1975SeriesDiscretes}. The geometrically constructed representations and the algebraically constructed representations have the same set of parameters when some regularity condition is satisfied, and Lusztig raised the problem on whether these two style representations actually coincide. In the case of even $r$, a solution to this problem was given for $\mathrm{GL}_n$ in \cite{ZheChen_PhDthesis}, which was generalised to the general case in \cite{ChenStasinski_2016_algebraisation}; we recall this result in Section~\ref{Deligne--Lusztig theory}. This result suggests that a character sheaf theory for $r>1$ may be established based on G\'{e}rardin's constructions.

\vspace{2mm} In Section~\ref{abelian groups}, we recall the character sheaves on abelian groups following \cite[5]{Lusz_CharSh_Generalizations}. In Section~\ref{generic character sheaves}, we define the generic character sheaves at even levels (i.e.\ $r$ even), and then compare their characteristic functions with the characters of even level Deligne--Lusztig representations. In Section~\ref{section:simple perversity} we state the conjecture on the simple perversity, and prove it for $r=2$ (see Theorem~\ref{main result}), which generalised a result of Lusztig in the function field case in \cite{Lusztig_2015_generic_CharSh}. In Section~\ref{ind and res}, we define the induction and the restriction functors, and study their transitive properties. In Section~\ref{frobenius reci}, we study the Frobenius reciprocity between the induction and the restriction functors. In the final Section~\ref{further remarks}, we give some remarks in the function field case. 

\vspace{2mm} Conventions: Denote by $k$ the residue field of $\mathcal{O}^{\mathrm{ur}}$. By a variety we always mean a reduced quasi-projective variety over $k$. For two elements $x$ and $y$ in an algebraic group, we use $x^y$ short for $y^{-1}xy$. For a product variety $X_1\times\cdots\times X_n$, we denote by $\pi_{X_i}$ the projection to $X_i$. For an algebraic group $G$ we denote by $G^o$ the identity component of $G$.

\vspace{2mm}\noindent {\bf Acknowledgement.} The author thanks Anne-Marie Aubert, Yongqi Feng, and Alexander Stasinski for helpful discussions. Some results in this work are in the author's Durham thesis. During the preparation of this work, the author is partially supported by the STU funding NTF17021.

\section{Generic characters and Deligne--Lusztig theory}\label{Deligne--Lusztig theory}

Let $\mathbb{G}$ be an affine smooth group scheme over $\mathcal{O}/\pi^r$, and let $\mathbf{G}$ be its base change to $\mathcal{O}^{\mathrm{ur}}/\pi^r$. There is an associated algebraic group $G=G_r=\mathcal{F}\mathbf{G}$ over $k$, where $\mathcal{F}$ is the Greenberg functor introduced in \cite{Greenberg19611} and \cite{Greenberg19632}. We recall some basic properties in this setting (see \cite{Sta2012ReductiveGr} and \cite{Sta2009Unramified} for the details): The Frobenius element in $\mathrm{Gal}(k/\mathbb{F}_q)$ gives a rational structure of $G$ over $\mathbb{F}_q$; we denote the associated geometric Frobenius by $F$. One has
\begin{equation*}
\mathbb{G}(\mathcal{O}/\pi^r)\cong G^F\quad \text{and} \quad \mathbf{G}(\mathcal{O}^{\mathrm{ur}}/\pi^r)\cong G(k)
\end{equation*}
as abstract groups. For any positive integer $i\leq r$, being the kernel of the reduction map modulo $\pi^i$ is a closed condition, hence defines a normal closed subgroup $G^i$ of $G$; we denote the quotient group $G/G^i$ by $G_i$. For convenience we put $G^0=G$. Similar notation also applies to closed subgroups of $G$. 

\vspace{2mm} From now on let $\mathbb{G}$ be a reductive group over $\mathcal{O}/\pi^r$; by this we mean it is an affine smooth group scheme over $\mathcal{O}/\pi^r$, with the geometric fibres being connected reductive groups in the usual sense. Let $\mathbf{T}$ be a maximal torus of $\mathbf{G}$ such that $T=\mathcal{F}\mathbf{T}$ is $F$-stable. Let $\mathbf{B}$ be a Borel subgroup of $\mathbf{G}$ containing $\mathbf{T}$; one has the Levi decomposition $\mathbf{B}=\mathbf{T}\mathbf{U}$, where $\mathbf{U}$ is the unipotent radical of $\mathbf{B}$. Let $B$ (resp.\ $U$) be the Greenberg functor image of $\mathbf{B}$ (resp.\ $\mathbf{U}$). The \emph{Deligne--Lusztig variety} associated to $T$ and $U$ is (see \cite{Lusztig2004RepsFinRings} and \cite{Sta2009Unramified})
\begin{equation*}
L^{-1}(FU):=\{g\in G\mid g^{-1}F(g)\in FU\}.
\end{equation*}
Note that $L^{-1}(FU)$ admits a left action of $G^F$ and a right action of $T^F$, and these two actions commute. They induce actions on the compactly supported cohomology groups $H^i_c(L^{-1}(FU),\overline{\mathbb{Q}}_{\ell})$ for every $i\in\mathbb{Z}$ (here $\ell$ is a fixed arbitrary prime not equal to $\mathrm{char}(k)$). For $\theta\in\widehat{T^F}=\mathrm{Hom}(T^F,\overline{\mathbb{Q}}_{\ell}^{\times})$, the $G^F$-module
\begin{equation*}
R_{T,U}^{\theta}=\sum_i(-1)^{i}H^i_c(L^{-1}(FU),\overline{\mathbb{Q}}_{\ell})_{\theta},
\end{equation*}
where the subscript $\theta$ means taking the $\theta$-isotypical part, is called a \emph{Deligne--Lusztig representation}. Note that this is a virtual representation over $\overline{\mathbb{Q}}_{\ell}$.

\vspace{2mm} In the remaining part of this paper we assume that $r=2l$ is even (do not mix $l$ with the prime $\ell$). Let $\mathbf{U}^{-}$ be the unipotent radical of the opposite Borel subgroup of $\mathbf{B}$, and denote by $U^{-}$ the corresponding Greenberg functor image. The commutative unipotent group $U^{\pm}:=(U^{-})^lU^l$ is called the arithmetic radical of $T$; it is $F$-stable (see \cite{ChenStasinski_2016_algebraisation}). For $\alpha\in\Phi$ a root of $\mathbf{T}$, we write $T^{\alpha}=\mathcal{F}\mathbf{T}^{\alpha}$, where $\mathbf{T}^{\alpha}$ denotes the image of the coroot $\check{\alpha}$. We call $W(T):=W(T_1)$ the Weyl group of $T$. More details on these concepts can be found in \cite{SGA3}. We recall some conditions on the characters of $T^F$  considered in \cite{Lusztig2004RepsFinRings}, \cite{Sta2009Unramified}, and \cite{ChenStasinski_2016_algebraisation}.

\begin{defi}
Let $\theta\in\widehat{T^F}$. 
\begin{itemize}
\item[(i)] If $\theta$ is not stabilised by any non-trivial element in $W(T)^F$, then $\theta$ is said to be in \emph{general position}.

\item[(ii)] Let $a$ be a positive integer such that $(T^{\alpha})^{r-1}$ is $F^a$-stable for every root $\alpha\in\Phi$. Then $\theta$ is said to be \emph{regular} if, for every root $\alpha\in\Phi$, there is some $t\in((T^{\alpha})^{r-1})^{F^a}$ such that $\theta(tF(t)...F^{a-1}(t))\neq1$. The regularity does not depend on the choice of $a$.

\item[(iii)] Let $\widetilde{\theta}$ be the trivial extension of $\theta$ to $(TU^{\pm})^F$, then $\theta$ is said to be \emph{generic} if it is regular, in general position, and $\mathrm{Stab}_{G^F}(\widetilde{\theta}|_{(G^l)^F})=(TU^{\pm})^F\cdot\mathrm{Stab}_{N(T)^F}(\widetilde{\theta}|_{(G^l)^F})$.
\end{itemize}
\end{defi}

We remark that the genericity is very close to the regularity, and it is actually a natural Lie algebra condition; see \cite[3.6 and 3.7]{ChenStasinski_2016_algebraisation}. 

\vspace{2mm} We have the following algebraisation result.

\begin{prop}\label{algebraisation theorem}
If $\theta\in \widehat{T^F}$ is regular and in general position, then $R_{T,U}^{\theta}$ is a subrepresentation of $\mathrm{Ind}_{(TU^{\pm})^F}^{G^F}\widetilde{\theta}$. If moreover $\theta$ is generic, then $R_{T,U}^{\theta}\cong\mathrm{Ind}_{(TU^{\pm})^F}^{G^F}\widetilde{\theta}$.
\end{prop}
\begin{proof}
See \cite{ChenStasinski_2016_algebraisation}.
\end{proof}

The representations of the form $\mathrm{Ind}_{(TU^{\pm})^F}^{G^F}\widetilde{\theta}$ were first studied by G\'{e}rardin \cite{Gerardin1975SeriesDiscretes}.

\section{Character sheaves on abelian groups}\label{abelian groups}

In this section we recall the setting of character sheaves on commutative algebraic groups in \cite[5]{Lusz_CharSh_Generalizations}, with a focus on $T$. 

\vspace{2mm} Fix an arbitrary positive integer $m$, let ${L'}$ be the Lang endomorphism associated to $F^m$ on $T$. There is a direct sum decomposition of $\overline{\mathbb{Q}}_{\ell}$-sheaves on $T$
\begin{equation*}
{L'}_*\overline{\mathbb{Q}}_{\ell}=\bigoplus_{\psi\in\widehat{T^{F^m}}}E^{\psi},
\end{equation*}
where $E^{\psi}$ is a locally constant $\overline{\mathbb{Q}}_{\ell}$-sheaf of rank $1$, whose stalk at $t\in T$ is the $1$-dimensional $\overline{\mathbb{Q}}_{\ell}$-representation space of $T^{F^m}$ given by $\psi$:
\begin{equation*}
E^{\psi}_t=\{f\colon {L'}^{-1}(t)\rightarrow \overline{\mathbb{Q}}_{\ell}\mid f(t_1t_2)=\psi(t_1)f(t_2), \forall t_1\in T^{F^m}, t_2\in {L'}^{-1}(t)\}.
\end{equation*}
There is a unique isomorphism of locally constant $\overline{\mathbb{Q}}_{\ell}$-sheaves $\varphi\colon (F^m)^*E^{\psi}\cong E^{\psi}$, such that at stalks $E^{\psi}_{F^m(t)}\rightarrow E^{\psi}_t$ it is $\varphi_t\colon f\mapsto f\circ F^m$. Note that if $t\in T^{F^m}$, then for $y\in {L'}^{-1}(t)$ one has $F^m(y)=ty$, so $f\circ F^m(y)=\psi(t)f(y)$ . 

\begin{defi}
Let $\mathcal{S}(T)$ be the set of all $E^{\psi}$, for various $m$ and $\psi$.
\end{defi}

\section{Generic character sheaves}\label{generic character sheaves}

The algebraisation of the generic even level Deligne--Lusztig representations (see Proposition~\ref{algebraisation theorem}) suggests that one can develop a generic character sheaf theory for reductive groups over $\mathcal{O}^{\mathrm{ur}}/\pi^r$, $r$ even, based on G\'{e}rardin's representations. Consider the diagram
\begin{equation}\label{basic diagram}
\begin{tikzcd}
T & Z_T \arrow{l}[swap]{b} \arrow{r}{a} & G
\end{tikzcd},
\end{equation}
where $Z_T:=\{(g, xTU^{\pm})\in G\times G/TU^{\pm}\mid g^x\in {TU^{\pm}}\}$; here $a$ is the natural left projection, and $b$ is the projection of $g^x\in TU^{\pm}$ to $\pi_T(g^x)\in T$ .

\begin{lemm}\label{rationality of Z_T and a,b}
The variety $Z_T$, as well as the morphisms $a$ and $b$, are all $F$-stable.
\end{lemm}
\begin{proof}
This follows from the fact that $U^{\pm}$ is $F$-stable.
\end{proof}

\begin{lemm}
The variety $Z_T$ is smooth and connected.
\end{lemm}
\begin{proof}
This can be proved in a way similar to \cite[2.5.2]{Lusztig_CharSh_I}. Consider the faithful flat base change $\widetilde{Z_T}$ of $Z_T\subseteq G\times G/TU^{\pm}$ along
$$G\times G\rightarrow G\times G/TU^{\pm}.$$ 
Then it suffices to show that $\widetilde{Z_T}=\{(g,x)\in G\times G\mid g^x\in TU^{\pm}\}\subseteq G\times G$ is smooth and connected (see \cite[17.7.7]{grothendieck_elements_1967}). By the variable change $b=g^x$ we get
\begin{equation*}
\widetilde{Z_T}\cong \{(b,x)\in G\times G\mid b\in TU^{\pm}\}=TU^{\pm}\times G,
\end{equation*}
which is smooth and connected.
\end{proof}

Given a variety $X$ over $k$, we write $\mathcal{D}(X)$ short for $D^b_c(X,\overline{\mathbb{Q}}_{\ell})$, the bounded derived category of constructible $\overline{\mathbb{Q}}_{\ell}$-sheaves constructed in \cite{deligne_conjecture_1980}.

\vspace{2mm} Back to the diagram \eqref{basic diagram}. For $m\in\mathbb{Z}_{>0}$ and $\theta\in\widehat{T^{F^m}}$, take $E^{\theta}\in\mathcal{S}(T)$, and put $K^{\theta}:=Ra_!(b^*E^{\theta})\in\mathcal{D}(G)$.

\begin{defi}
We call the complexes $K^{\theta}$, for various generic characters $\theta\in \widehat{T^{F^m}}$ and various $m\in\mathbb{Z}_{>0}$, the \emph{generic character sheaves} on $G$.
\end{defi}

Let $\theta\in\widehat{T^{F^m}}$ be a generic character. The isomorphism $\varphi\colon (F^{m})^*E^{\theta}\cong E^{\theta}$ induces an isomorphism $\varphi\colon(F^m)^*K^{\theta}\cong K^{\theta}$. We want to evaluate the \emph{characteristic function} of the complex $K^{\theta}$ with respect to $\varphi$, i.e.\ evaluate
\begin{equation*}
\chi_{K^{\theta},\varphi}(g)=\sum_{i\in\mathbb{Z}}(-1)^i\cdot\mathrm{Tr}(\varphi_g,\mathcal{H}^i(K^{\theta})_g)
\end{equation*}
for $g\in G^{F^m}$:

\begin{prop}
Along with the above notation, we have $$\chi_{K^{\theta},\varphi}(g)=\mathrm{Tr}(g,R_{T,U}^{\theta})$$
for any $g\in G^{F^m}$. (Here $R_{T,U}^{\theta}$ is defined with respect to $F^m$, not $F$.)
\end{prop}
\begin{proof}
This argument is standard. First, by the proper base change (of $a$ along the injection $\{g\}\rightarrow G$) we have
\begin{equation}\label{4a}
\begin{split}
\chi_{Ra_!b^*E^{\theta},\varphi}(g)
&=\sum_{i\in\mathbb{Z}}(-1)^i\cdot\mathrm{Tr}\left(\varphi_g,\mathcal{H}^i(Ra_!(b^*E^{\theta}))_g\right)\\
&=\sum_{i\in\mathbb{Z}}(-1)^i\cdot\mathrm{Tr}\left(\varphi,{H}_c^i(a^{-1}(g),b^*E^{\theta})\right).
\end{split}
\end{equation}
By applying Grothendieck's Lefschetz trace formula (see \cite[Rapport-3.2]{SGA412}) we get
\begin{equation*}
\begin{split}
\eqref{4a}
&=\sum_{xTU^{\pm}\in (G/TU^{\pm})^{F^m},\ (g,xTU^{\pm})\in Z_T}\mathrm{Tr}\left(\varphi,b^*E^{\theta}_{(g,xTU^{\pm})}\right)\\
&=\sum_{xTU^{\pm}\in (G/TU^{\pm})^{F^m},\ g^x\in TU^{\pm}}\mathrm{Tr}\left(\varphi,E^{\theta}_{\pi_T(g^x)}\right)\\
&=\frac{1}{|(TU^{\pm})^{F^m}|}\cdot\sum_{x\in G^{F^m},\ g^x\in TU^{\pm}}\theta\left(\pi_T(g^x)\right).
\end{split}
\end{equation*}
This is, by definition of induced characters, the character value of $\mathrm{Ind}_{(TU^{\pm})^{F^m}}^{G^{F^m}}\widetilde{\theta}$ at $g$. Now the assertion follows from Proposition~\ref{algebraisation theorem}.
\end{proof}

\section{Simple perversity}\label{section:simple perversity}

For a variety $X$ over $k$, let $\mathcal{D}^{\leq0}(X)$ be the full subcategory of $\mathcal{D}(X)$ consisting of the objects $K$ satisfying: The support of $\mathcal{H}^i(K)$ has dimension $\leq -i$ for any integer $i$ (in particular, $\mathcal{H}^i(K)=0$ if $i>0$, and the support of $\mathcal{H}^0(K)$ is a finite set). Meanwhile, let $\mathcal{D}^{\geq0}(X)$ be the full subcategory of $\mathcal{D}(X)$ consisting of the objects $K$ such that $\mathbb{D}_X(K)\in\mathcal{D}^{\leq0}(X)$, where $\mathbb{D}_X$ denotes the Verdier duality functor on $\mathcal{D}(X)$. The category of perverse sheaves is the full subcategory $\mathcal{M}(X):=\mathcal{D}^{\leq0}(X)\cap\mathcal{D}^{\geq0}(X)$.

\begin{defi}
Let $\theta\in\widehat{T^F}$. We call $\theta$ geometrically generic, if it extends to a generic character of the abelian group $T^{F^m}$ for every $m\in\mathbb{Z}_{>0}$.
\end{defi}

We remark that (i) it seems that geometric genericity actually coincides with genericity, and (ii) when $\mathrm{char}(\mathbb{F}_q)$ is big enough, being geometrically generic equivalents to corresponding to a regular semisimple element in the Lie algebra; see \cite[3.6]{ChenStasinski_2016_algebraisation} for more details.

\vspace{2mm} Now we can state the conjecture (due to Lusztig when $\mathcal{O}=\mathbb{F}_q[[\pi]]$):

\begin{conj}
If $\theta$ is geometrically generic, then $K^{\theta}[\dim G]$ is a simple perverse sheaf.
\end{conj}

In the below, extending the idea of Lusztig in the function field case (in \cite{Lusztig_2015_generic_CharSh}), we prove this conjecture in the level two case, by showing that Deligne's $\ell$-adic Fourier transformation of $K^{\theta}[\dim G]$ is a simple perverse sheaf supported on a subvariety (of $G$) defined by the regular semisimple orbit (in the Lie algebra) corresponding to $\theta$.

\begin{thm}\label{main result}
The above conjecture is true for $r=2$ and $\mathrm{char}(k)$ very good for $G_1$.
\end{thm}
\begin{proof}
First, $G^1$ is a connected unipotent group isomorphic to the additive group of the Lie algebra of $G_1$ (see \cite{Greenberg19632}), hence a special group in the sense of \cite{serre_1958_special_gp}, thus $G\rightarrow G_1$ is a vector bundle of rank $\dim G^1$. We fix a local trivialisation for this principal vector bundle, and for each $g\in G$ we assign to it a fixed local piece in this trivialisation (so we can talk about the $G^1$-part and the $G_1$-part of $g$). Within this vector bundle, we will use the $\ell$-adic Fourier transform technique in \cite{Laumon_1987_Transformation_Fourier}: As $\mathrm{char}(k)$ is very good, there is a non-degenerate $G$-invariant symmetric bilinear form on $G^1$ (see e.g.\ \cite[2.5.12]{Let2005book}), which identifies the dual of $G^1$ to itself, hence produce a pairing $h(-,-)\colon G\times_{G_1} G\rightarrow \mathbb{A}^1$. The geometric genericity of $\theta$ then means there is a regular semisimple element $t_{\theta}\in (T^F)^1$, such that $\widetilde{\theta}(s)=\psi(h(t_{\theta},s))$ for all $s\in (G^1)^F$, where $\psi\colon \mathbb{F}_q\rightarrow\overline{\mathbb{Q}}_{\ell}^{\times}$ is a fixed non-trivial group morphism (\cite[3.6]{ChenStasinski_2016_algebraisation}).

\vspace{2mm} Identifying $T_1$ as the reductive part of $T$, we get $T=T_1\times T^1$; let $\mathcal{E}^{\theta}$ be the character sheaf on $T_1$ associated to $\theta|_{T_1^F}$ (in the sense of \cite[0.5]{Lusz_CharSh_Generalizations}). Meanwhile, for a morphism $f\colon X\rightarrow\mathbb{A}^1$ over $\overline{\mathbb{F}}_q$, we denote the Artin--Schreier sheaf associated to $f$ and $\psi$ by $\mathcal{L}_f$ (see e.g.\ \cite[0.3]{Lusztig_2015_generic_CharSh}); in particular, we have a local system $\mathcal{L}_{h_{\theta}}$ on $T^1$, where $h_{\theta}$ is the restriction of $h(t_{\theta},-)$ to $T^1$. Denote the pull-backs of $\mathcal{E}^{\theta}$ and $\mathcal{L}_{h_{\theta}}$ to $Z_T$ (along $b$) by $\mathcal{E}$ and $\mathcal{L}$, respectively, then we need to show that  $Ra_!(\mathcal{E}\otimes\mathcal{L})[\dim G]$ is a simple perverse sheaf. 

\vspace{2mm} Let $\pi_1$ and $\pi_2$ be the first and the second projections $G\times_{G_1}G\rightarrow G$, then by Laumon's theorem \cite[1.3.2]{Laumon_1987_Transformation_Fourier} it suffices to show that the Fourier transform of $K^{\theta}[\dim G]$, that is,
$$\hat{K}:=R(\pi_2)_!(\pi_1^*Ra_!(\mathcal{E}\otimes\mathcal{L})\otimes \mathcal{L}_h)[3\dim G^1],$$ 
is simple perverse.

\vspace{2mm} Consider the pull-back $a'$ of $a\colon Z_T\rightarrow G$ along $\pi_1$
\begin{equation*}
\begin{tikzcd}
Z_T' \arrow{d}[swap]{a'} \arrow{r}{\pi_1'} & Z_T \arrow{d}{a} \\
G\times_{G_1}G \arrow{r}{\pi_1}& G.
\end{tikzcd}
\end{equation*}
By proper base change and projection formula we get 
\begin{equation*}
\begin{split}
\hat{K}
&=R(\pi_2)_!(R(a')_!\pi_1'^*(\mathcal{E}\otimes\mathcal{L})\otimes \mathcal{L}_h)[3\dim G^1]\\
&=R(\pi_2)_!(R(a')_!(\pi_1'^*(\mathcal{E}\otimes\mathcal{L})\otimes a'^*\mathcal{L}_h))[3\dim G^1]\\
&=R(\pi_2)_!(R(a')_!(\mathcal{E}'\otimes\mathcal{L}'\otimes a'^*\mathcal{L}_h))[3\dim G^1]\\
&=R(\pi_2\circ a')_!(\mathcal{E}'\otimes\mathcal{L}_{h'})[3\dim G^1],
\end{split}
\end{equation*}
where $\mathcal{E}'$ and $\mathcal{L}'$ denote the pull-backs (along $\pi_1'$) of $\mathcal{E}$ and $\mathcal{L}$ to $Z_T'$, respectively, and $h'$ denotes the morphism $ h\circ a'+h_{\theta}\circ \pi_{T^1} \circ b\circ \pi_1' $ (here $\pi_{T^1}$ is the projection from $T$ to $T^1$).

\vspace{2mm} Note that $Z_T'$ can be written as $\{ (y,x)\in (G\times_{G_1}G)\times G/TU^{\pm} \mid \pi_1(y)^x\in TU^{\pm} \}$. As $\pi_1(y)^x\in TU^{\pm}$ is actually a condition only concerns the $G_1$-part of $\pi_1(y)$, we see that $\pi_2\times \mathrm{Id}_{G/TU^{\pm}}$ gives $Z_T'$ a vector bundle structure over $Z_T$ (with fibre $\cong G^1$); denote the restriction of $\pi_2\times \mathrm{Id}_{G/TU^{\pm}}$ to $Z_T'$ by $\mu$. Given $(g,x)\in Z_T$, identify $\mu^{-1}(g,x)$ as $G^1$ via the fixed local trivialisation, then the restriction of $h'$ on $\mu^{-1}(g,x)$ can be written as 
\begin{equation*}
h'\colon (y,x)\mapsto h(y)+ h(t_{\theta},\pi_{T^1}\circ b(\pi_1(y),x))=h({\pi_1(y)\tilde{}},{g\tilde{}})+ h((^xt_{\theta}),\pi_x\pi_1(y)),
\end{equation*}
where ${\pi_1(y)\tilde{}}$ and ${g\tilde{}}$ denote the $G^1$-parts of $\pi_1(y)$ and $g$, respectively, and $\pi_x$ denotes the projection from ${^x(TU^{\pm})}$ to ${^x(T^1)}$. However, as $h$ is symmetric, by the definition of $\widetilde{\theta}$ we see $h'(y,x)=h({g\tilde{}}\cdot(^xt_{\theta}),{\pi_1(y)\tilde{}})$. Therefore, $h'$ is zero on $\mu^{-1}(g,x)$ if and only if ${g\tilde{}}=(^xt_{\theta})^{-1}$; note that this is a closed condition, so all such $(g,x)$ form a closed subvariety $\bar{Z}\subseteq Z_T$. 

\vspace{2mm} Denote the pull-back of $\mathcal{E}$ to $\bar{Z}$ by $\bar{\mathcal{E}}$, and denote the restriction of $a$ to $\bar{Z}$ by $\bar{a}$. Since $\pi_2\circ a'=a\circ\mu$, according to \cite[1.3(b)]{Lusztig_2015_generic_CharSh} we have
\begin{equation*}
\hat{K}\cong R(\bar{a})_!(\bar{\mathcal{E}})[\dim G^1].
\end{equation*}
Meanwhile, note that $\bar{a}$ defines an isomorphism from $\bar{Z}$ to a subvariety of $G$, which is locally
\begin{equation*}
\{ (g',g'')\in V\times\mathcal{C}\mid g'\in\mathrm{Stab}_G(g'')^{o} \},
\end{equation*}
where $V\subseteq G_1$ is a local piece in the local trivialisation and $\mathcal{C}$ is the (adjoint) $G$-orbit of $(t_{\theta})^{-1}$. Recall that $\mathcal{C}\subseteq G^1$ is a closed subvariety (see \cite[III.9]{Borel_1991_LinearAlgGp}), and by the definition of regular elements it is of dimension $\dim G_1-\dim T_1$, so $\bar{Z}\subseteq G$ is a smooth closed subvariety of dimension $\dim G^1$ via $\bar{a}$ (for the smoothness, consider the projection from $\bar{Z}$ to $\mathcal{C}$). This implies that $\hat{K}\cong \bar{a}_!\bar{\mathcal{E}}[\dim G^1]$ is a simple perverse sheaf on $G$, which completes the proof.
\end{proof}

In the rest of this paper we turn to the discussion of the induction and restriction functors.

\begin{remark}\label{remark: basic properties of D(X) and M(X)}
Here we collect some basic properties of $\mathcal{D}(X)$ and $\mathcal{M}(X)$, which will be used later; the details and further properties can be found in \cite{BBD}, \cite{Lusztig_CharSh_I}, \cite{Lusztig_1994_IntroQuantum}, and \cite{Boyarchenko_Drinfeld2010CharSh}. Let $f\colon X\rightarrow Y$ be a morphism of varieties over $k$. We use $f^*$, $f^!$, $Rf_*$, and $Rf_!$ to denote the corresponding derived functors.
\begin{itemize}
\item[(P1)] Suppose that $f$ is smooth with connected fibres of dimension $d$. Then (ignore Tate twists) $f^!=f^*[2d]$. Moreover, $\widetilde{f}:=f^*[d]$ is fully faithful from $\mathcal{M}(Y)$ to $\mathcal{M}(X)$. If $K\in\mathcal{D}(Y)$, then $\widetilde{f}K\in\mathcal{D}^{\geq0}(X)$ (resp.\ $\in\mathcal{D}^{\leq0}(X)$) if and only if $K\in\mathcal{D}^{\geq0}(Y)$ (resp.\ $\in\mathcal{D}^{\leq0}(Y)$). If $K\in\mathcal{D}^{\leq0}(Y)$ and $K'\in\mathcal{D}^{\geq0}(Y)$, then $\mathrm{Hom}_{\mathcal{D}(Y)}(K,K')=\mathrm{Hom}_{\mathcal{D}(X)}(\widetilde{f}K,\widetilde{f}K')$. 

\item[(P2)] Suppose that $h\colon H\times Y\rightarrow Y$ is the action morphism of a connected algebraic group $H$ acting on $Y$. Let $\pi_Y\colon H\times Y\rightarrow Y$ be the projection morphism. Note that both $h$ and $\pi_Y$ are smooth of dimension $\dim H$. We say that $K\in\mathcal{M}(Y)$ is $H$-equivariant, if the two perverse sheaves $\widetilde{h}K$ and $\widetilde{\pi_Y}K$ are isomorphic.

\item[(P3)] Suppose that $H$ is a connected algebraic group acting freely on $X$ and trivially on $Y$, and that $f$ is locally trivial and $H$-equivariant. If $K\in \mathcal{M}(X)$, then $K$ is $H$-equivariant if and only if $K\cong\widetilde{f}K'$ for some $K'\in\mathcal{M}(Y)$.
\end{itemize}
\end{remark}

\section{Induction and restriction}\label{ind and res}

In this section we define the induction functors (from the equivariant perverse sheaves on the Levi subgroups to $\mathcal{D}(G)$) and the restriction functors (from the derived category $\mathcal{D}(G)$ to that of the Levi subgroups), and then study their transitive properties.

\vspace{2mm} Fix a parabolic subgroup $\mathbf{P}$ of $\mathbf{G}$, and let $\mathbf{M}$ be a Levi subgroup of $\mathbf{P}$ (when there is no confusion we will say that $\mathbf{M}$ is a Levi subgroup of $\mathbf{G}$); denote by $P$ and $M$ the corresponding Greenberg functor images, respectively. Write $\mathbf{U}_{\mathbf{P}}$ for the unipotent radical of $\mathbf{P}$, and write $\mathbf{U}_{\mathbf{P}}^{-}$ for the unipotent radical of the opposite parabolic subgroup; denote their Greenberg functor images by $U_P$ and $U^{-}_P$, respectively. We put $U^{\pm}_{G-M}:=(U_P)^l(U^{-}_P)^l$ and $U^{\pm}_M:=M\cap U^{\pm}$. For detailed properties of parabolic subgroups and Levi subgroups we refer to \cite[XXVI]{SGA3}.

\vspace{2mm} Now we are going to define the induction functor, which requires some technical preparations.

\vspace{2mm} Consider the varieties
\begin{equation*}
Z^G_M:=\{(g, xMU^{\pm})\in G\times G/MU^{\pm}\mid g^x\in {MU^{\pm}}\}
\end{equation*}
and 
\begin{equation*}
\widetilde{Z^G_M}:=\{(g, x)\in G\times G\mid g^x\in {MU^{\pm}}\}\cong MU^{\pm}\times G;
\end{equation*}
they admit the $G$-action $y\in G\colon (g,x)\mapsto (ygy^{-1},yx)$. Consider the $G$-equivariant smooth morphism with connected fibres
\begin{equation*}
\pi'_{M,G}\colon\widetilde{Z^G_M}\longrightarrow Z^G_M,\quad (g,x)\mapsto (g,xMU^{\pm});
\end{equation*}
this is a principal $MU^{\pm}$-fibration, where $MU^{\pm}$ acts on $\widetilde{Z^G_M}$ by $y'\in MU^{\pm}\colon(g,x)\mapsto (g,xy'^{-1})$. Moreover, note that the quotient $G\rightarrow G/MU^{\pm}$ is locally trivial: As $G^l$ is a connected unipotent group, it suffices to show that $G_l\rightarrow G_l/M_l$ is locally trivial, which follows from the fact that the multiplication morphism $\mathbf{U}_{\mathbf{P}}^{-}\times\mathbf{U}_{\mathbf{P}}\times\mathbf{M}\rightarrow\mathbf{G}$ is an open immersion (see \cite[XXVI 4.3.2]{SGA3}). Therefore $\pi'_{M,G}$ is a locally trivial principal fibration by $MU^{\pm}$. Consider the trivial $G$-action on $M$, then we get another $G$-equivariant smooth morphism with connected fibres
\begin{equation*}
\pi_{M,G}\colon \widetilde{Z^G_M}\longrightarrow M,\quad (g,x)\mapsto\pi_M(g^x),
\end{equation*}
where $\pi_M$ is the projection from $MU^{\pm}$ to $M$. Note that the action of $MU^{\pm}$ on $\widetilde{Z^G_M}$ induces an action of $MU^{\pm}$ on $M$, which is compatible with the conjugation action of $M$ on $M$, so $\pi_{M,G}$ is also $MU^{\pm}$-equivariant. Now we get a diagram
\begin{equation}\label{induction diagram}
\begin{tikzcd}
M & \widetilde{Z^G_M}\arrow{l}[swap]{\pi_{M,G}} \arrow{r}{\pi'_{M,G}} & Z_M^G  \arrow{r}{\pi''_{M,G}} & G,
\end{tikzcd}
\end{equation}
where $\pi''_{M,G}$ is the left projection (which is $G$-equivariant with respect to the conjugation action of $G$ on itself). 

\vspace{2mm} Consider \eqref{induction diagram}, if $K\in\mathcal{M}(M)$ is $M$-equivariant (see Remark~\ref{remark: basic properties of D(X) and M(X)} (P2)) with respect to the conjugation action, then from the above we see that $\widetilde{\pi_{M,G}}K$ is $MU^{\pm}$-equivariant and $G$-equivariant. Moreover, since $\pi'_{M,G}$ is a locally trivial principal fibration by $MU^{\pm}$, there is a unique (up to isomorphisms) perverse sheaf (see Remark~\ref{remark: basic properties of D(X) and M(X)} (P3))
\begin{equation}\label{define (-)_{M,G}}
K_{M,G}\in\mathcal{M}(Z^G_M)
\end{equation} 
such that $\widetilde{\pi_{M,G}}K\cong \widetilde{\pi'_{M,G}}K_{M,G}$. Note that by Remark~\ref{remark: basic properties of D(X) and M(X)} (P1) and the $G$-equivariance of $\widetilde{\pi_{M,G}}K$, this $K_{M,G}$ is $G$-equivariant.

\begin{defi}
Given $K\in\mathcal{M}(M)$ equivariant with respect to the conjugation action of $M$, along with the above notation, we put $\mathrm{ind}_M^GK:=R(\pi''_{M,G})_!K_{M,G}\in\mathcal{D}(G)$.
\end{defi}

Note that if $\mathrm{ind}_M^GK\in\mathcal{D}(G)$ is perverse, then it is $G$-equivariant by proper base change (\cite[1.7.5]{Lusztig_CharSh_I}) and the $G$-equivariance of $K_{M,G}$.

\begin{expl}
If $\mathbf{M}=\mathbf{G}$, then $\mathrm{ind}_G^GK=K$ by the equivariant property.
\end{expl}

\begin{expl}
If $\mathbf{M}=\mathbf{T}$ is the maximal torus, then $Z_T^G=Z_T$, and $\pi_{M,G}$ naturally factors through $b\colon Z^G_T\rightarrow T$ ($b$ is the morphism in \eqref{basic diagram}), which implies $K_{M,G}\cong\widetilde{b}K$. Thus in this simpler situation, we see that $\mathrm{ind}_T^GE^{\theta}[\dim T]\cong K^{\theta}[\dim G]$ for any ${\theta}\in\widehat{T^F}$.
\end{expl}

\begin{prop}\label{transitivity for induction}
Let $\mathbf{N}$ be a Levi subgroup of $\mathbf{M}$; denote by $N$ its Greenberg functor image. If $K\in\mathcal{M}(N)$ is an $N$-equivariant perverse sheaf such that $\mathrm{ind}_N^MK$ is a perverse sheaf, then
\begin{equation*}
\mathrm{ind}^G_N K\cong \mathrm{ind}^G_M\circ\mathrm{ind}^M_NK.
\end{equation*}
\end{prop}
\begin{proof}
The argument is an analogue to the one in \cite[4.2]{Lusztig_CharSh_I}. We have a commutative diagram
\begin{equation*}
\begin{tikzcd}
&N&&\\
\widetilde{Z_N^M}\arrow{ru}{\pi_{N,M}} \arrow{d}[swap]{\pi'_{N,M}} & X  \arrow{l}[swap]{h_1} \arrow{u}[swap]{d}\arrow{r}{h_2}\arrow{d}{f} & \widetilde{Z_N^G}\arrow{lu}[swap]{\pi_{N,G}}\arrow{d}{\pi'_{N,G}} &\\
Z_N^M \arrow{d}[swap]{\pi''_{N,M}} & Y\arrow{l}[swap]{e_1}\arrow{r}{e_2} \arrow{d}{g_1}  & Z_N^G \arrow{d}{g_2}  \arrow{dr}{\pi''_{N,G}} &\\
M & \widetilde{Z_M^G}\arrow{l}[swap]{\pi_{M,G}}\arrow{r}{\pi'_{M,G}} & {Z_M^G}\arrow{r}{\pi''_{M,G}} & G,
\end{tikzcd}
\end{equation*}
where $X:=\{(g,x,z)\in G\times G\times MU^{\pm}\mid g^x\in NU^{\pm}\}$, and $Y$ is the quotient of $X$ by the $NU^{\pm}$-action given by $q\in NU^{\pm}\colon (g,x,z)\mapsto (g,xq^{-1},zq^{-1})$; $f$ denotes the quotient morphism. The other morphisms are as below:
\begin{itemize}
\item[] $d\colon (g,x,z)\mapsto \pi_N(g^x)$, where $\pi_N$ is the projection from $NU^{\pm}$ to $N$;

\item[] $h_1$ is $(g,x,z)\mapsto (\pi_M(g^{xz^{-1}}),\pi_M(z))$;

\item[] $h_2\colon (g,x,z)\mapsto (g,x)$;

\item[] $e_1\colon (g,x,z)\mapsto (\pi_M(g^{xz^{-1}}),\pi_M(z)NU^{\pm}_M)\in  Z_N^M\subseteq M\times M/NU_M^{\pm}$;

\item[] $e_2\colon (g,x,z)\mapsto (g,xNU^{\pm})\in  Z_N^G=\{(g,x)\in G\times G/NU^{\pm}\mid g^x\in NU^{\pm} \}$;

\item[] $g_1\colon (g,x,z)\mapsto (g,xz^{-1})$;

\item[] $g_2\colon (g,xNU^{\pm})\mapsto (g,xMU^{\pm})$.
\end{itemize}
Note in the above diagram, the two bottom squares are cartesian, and $e_i$ and $f$ are smooth morphisms with connected fibres.

\vspace{2mm} To show that $\mathrm{ind}_N^GK\cong\mathrm{ind}_M^G\mathrm{ind}_N^MK$, in other words, to show that 
\begin{equation*}
R(\pi''_{M,G})_!R(g_2)_!K_{N,G}\cong R(\pi''_{M,G})_!(\mathrm{ind}_N^MK)_{M,G},
\end{equation*}
it suffices to show that
\begin{equation*}
R(g_2)_!K_{N,G}\cong(\mathrm{ind}_N^MK)_{M,G}.
\end{equation*}
(Here, for a perverse sheaf $A$, the notation $A_{B,C}$ is defined in the same way as in \eqref{define (-)_{M,G}}, by formally replacing $M$ and $G$ by $B$ and $C$, respectively.) Since $\widetilde{\pi'_{M,G}}$ is fully faithful on perverse sheaves (see Remark~\ref{remark: basic properties of D(X) and M(X)} (P1)), this assertion can be deduced by showing that
\begin{equation}\label{temp4}
\widetilde{\pi'_{M,G}}R(g_2)_!K_{N,G}\cong\widetilde{\pi'_{M,G}}(\mathrm{ind}_N^MK)_{M,G}.
\end{equation}
Note that
\begin{equation*}
\widetilde{\pi'_{M,G}}(\mathrm{ind}_N^MK)_{M,G}\cong\widetilde{\pi_{M,G}}\mathrm{ind}_N^MK\cong \widetilde{\pi_{M,G}} R({\pi}''_{N,M})_!K_{N,M}
\end{equation*}
by the definition of $(\mathrm{ind}_N^MK)_{M,G}$, so \eqref{temp4} is equivalent to
\begin{equation*}
R(g_1)_!\widetilde{e_2}K_{N,G}\cong R(g_1)_!\widetilde{e_1}K_{N,M}
\end{equation*}
by applying the proper base change (on both sides). Thus we only need to show that $\widetilde{e_2}K_{N,G}\cong\widetilde{e_1}K_{N,M}$, which is equivalent to showing
\begin{equation}\label{temp5}
\widetilde{f}\widetilde{e_2}K_{N,G}\cong\widetilde{f}\widetilde{e_1}K_{N,M}
\end{equation}
(by the full faithfulness of $\widetilde{f}$ on perverse sheaves). By the definitions of $K_{N,M}$ and $K_{N,G}$, \eqref{temp5} follows from
\begin{equation*}
\widetilde{h_1}\widetilde{\pi_{N,M}}K=\widetilde{d}K=\widetilde{h_2}\widetilde{\pi_{N,G}}K.
\end{equation*}
This completes the proof.
\end{proof}

Now we turn to the restriction functor.

\begin{defi}
Consider the diagram
\begin{equation}\label{restriction diagram}
\begin{tikzcd}
M & MU^{\pm} \arrow{l}[swap]{\pi_M} \arrow{r}[hook]{i} & G,
\end{tikzcd}
\end{equation}
where $i$ is the natural closed immersion and $\pi_M$ is the projection from $MU^{\pm}\cong M\times U^{\pm}_{G-M}$ to $M$. For any $K\in\mathcal{D}(G)$, we put $\mathrm{res}^G_MK:=R(\pi_M)_!i^*K\in\mathcal{D}(M)$. 
\end{defi}

\begin{prop}\label{transitivity for restriction}
Suppose that $\mathbf{N}$ is a Levi subgroup of $\mathbf{M}$, and denote by $N$ its Greenberg functor image, then
\begin{equation*}
\mathrm{res}^G_N\cong \mathrm{res}^M_N\circ\mathrm{res}^G_M.
\end{equation*}
\end{prop}
\begin{proof}
Consider the following commutative diagram
\begin{equation*}
\begin{tikzcd}
N && NU^{\pm}\arrow{ll}[swap]{\pi_N} \arrow{rr}[hook]{i} \arrow{ld}[swap]{\pi'}\arrow{rd}[hook]{i'_2} && G  \\
& NU^{\pm}_M\arrow{lu}{\pi'_N}\arrow{rd}[hook,swap]{i_2} && MU^{\pm}\arrow{ld}{\pi_M}\arrow{ur}[hook,swap]{i_1} & \\
&& M &&,
\end{tikzcd}
\end{equation*}
Where $\pi'$ and $\pi'_N$ are the natural projections, and $i_1$, $i_2$, and $i'_2$ are the natural inclusions. Note that the middle diamond is cartesian, so by the proper base change theorem we have
\begin{equation*}
\begin{split}
\mathrm{res}^G_N
&=R(\pi_N)_!i^*\\
&=R(\pi'_N)_!R(\pi')_!(i'_2)^*i_1^*\\
&=R(\pi'_N)_!(i_2)^*R(\pi_M)_!i_1^*\\
&=\mathrm{res}^M_N\circ\mathrm{res}^G_M.
\end{split}
\end{equation*}
Thus the transitive property holds.
\end{proof}

\section{Frobenius reciprocity}\label{frobenius reci}

In this section we will be concerned with a Frobenius reciprocity. 

\vspace{2mm} In the level $r=1$ case, if $A$ is a character sheaf, then Lusztig proved that $\mathrm{res}^G_MA\in\mathcal{D}^{\leq0}(M)$ \cite[4.4 (c)]{Lusztig_CharSh_I}, and based on this property he established the Frobenius reciprocity (see \cite[4.4]{Lusztig_CharSh_I}). We expect that our generic character sheaves also satisfy this property (after a shifting). However, the method in the $r=1$ case does not work well in the $r$ even case; one of the obstructions is that the morphism $a$ in \eqref{basic diagram} is not proper. In any case, in the below we show that, when such a property holds, there is still a Frobenius reciprocity for $r$ even:

\begin{prop}\label{Frobenius formula}
If $A_1\in\mathcal{M}(M)$ is $M$-equivariant, and if $A\in\mathcal{M}(G)$ is $G$-equivariant such that $\mathrm{res}^G_MA\in\mathcal{D}^{\leq0}(M)$, then
\begin{equation*}
\mathrm{Hom}_{\mathcal{D}(M)}(\mathrm{res}^G_MA,A_1)\cong\mathrm{Hom}_{\mathcal{D}(G)}(A,\mathbb{D}_G\circ\mathrm{ind}_M^G\circ\mathbb{D}_MA_1).
\end{equation*}
\end{prop}

\begin{proof}
We combine the Verdier duality with the methods in \cite[4.4]{Lusztig_CharSh_I}. Consider the commutative diagram
\begin{equation*}
\begin{tikzcd}
& {Z_M^G} \arrow{dl}[swap]{f}\arrow{r}{\pi''_{M,G}} & G    & \\
MU^{\pm}\backslash M\times G && MU^{\pm}\times G \arrow{lu}[swap]{\rho}\arrow{ld}[swap]{\phi}\arrow{rd}{\theta'}\arrow{r}{\zeta'}\arrow{u}[swap]{\zeta} \arrow{dd}{\theta}  & G   \\
& M\times G\arrow{lu}{\beta}\arrow{rd}[swap]{\gamma} &  & MU^{\pm}\arrow{ld}{\pi_M}\arrow{u}[hook,swap]{i}\\
&&M&,
\end{tikzcd}
\end{equation*}
where $i$ and $\pi_M$ are as in \eqref{restriction diagram}, $\pi''_{M,G}$ is as in \eqref{induction diagram}, and $\beta$ is the quotient morphism of the $MU^{\pm}$-action on $M\times G$ given by
\begin{equation*}
y\in MU^{\pm}\colon (g,x)\mapsto (\pi_M(y)g\pi_M(y)^{-1},xy^{-1});
\end{equation*}
and the other morphisms are as below:
\begin{itemize}
\item[] $\zeta'$, $\theta'$, and $\gamma$ are left projections;

\item[] $\theta\colon (g,x)\mapsto \pi_M(g)$;

\item[] $\phi:=\pi_M\times\mathrm{id}$;

\item[] $\zeta\colon (g,x)\mapsto xgx^{-1}$.
\end{itemize}
Concerning $f$ and $\rho$, first recall that $Z_M^G=\{(g,x)\in G\times G/MU^{\pm}\mid g^x\in MU^{\pm}\}$; we put
\begin{itemize}
\item[] $f\colon (g,xMU^{\pm})\mapsto (\pi_M(g^x),x)\mod MU^{\pm}$;

\item[] $\rho\colon (g,x)\mapsto (xgx^{-1},xMU^{\pm})$.
\end{itemize}
Note that in this diagram, by identifying $MU^{\pm}\times G$ with $\widetilde{Z_{M}^G}$ we see that $\rho$ becomes $\pi'_{M,G}$, so ($\phi$, $\gamma$, $\pi_M$, $\theta'$) and ($f$, $\beta$, $\rho$, $\phi$) form two cartesian diagrams. Also note that $\beta$ is a locally trivial fibration (as $\phi$ and $\rho$ are locally trivial fibrations; see the arguments above \eqref{induction diagram}), and $f$ is smooth with connected fibres ($\cong U^{\pm}_{G-M}$).

\vspace{2mm} We have (Remark~\ref{remark: basic properties of D(X) and M(X)} (P1))
\begin{equation}\label{temp1}
\mathrm{Hom}_{\mathcal{D}(M)}\left(\mathrm{res}^G_MA,A_1\right)\cong\mathrm{Hom}_{\mathcal{D}(M\times G)}\left(\widetilde{\gamma}\mathrm{res}^G_MA,\widetilde{\gamma}A_1\right).
\end{equation}
Consider the right hand side; by the proper base change theorem we see that
\begin{equation*}
\begin{split}
\widetilde{\gamma}\mathrm{res}^G_MA
=&\gamma^*R(\pi_M)_!i^*A[\dim G]\\
=&R\phi_!(\theta')^*i^*A[\dim G]\\
=&R\phi_!(\zeta')^*A[\dim G],
\end{split}
\end{equation*}
which is actually $R\phi_!(\zeta)^*A[\dim G]$ by the equivariance of $A$, and then again by the proper base change we get
\begin{equation*}
\begin{split}
R\phi_!(\zeta)^*A[\dim G]
=&R\phi_!\rho^*(\pi''_{M,G})^*A[\dim G]\\
=&\beta^*Rf_!(\pi''_{M,G})^*A[\dim G]\\
=&\widetilde{\beta}Rf_!(\pi''_{M,G})^*A[\dim U^{\pm}_{G-M}].
\end{split}
\end{equation*}
On the other hand, since $\gamma$ is $MU^{\pm}$-equivariant with respect to the conjugation action of $MU^{\pm}$ composed by $\pi_M$, we see that $\widetilde{\gamma}A_1$ is $MU^{\pm}$-equivariant, thus $\widetilde{\gamma}A_1=\widetilde{\beta}A'_1$ for some $A'_1\in\mathcal{M}(MU^{\pm}\backslash M\times G)$ (Remark~\ref{remark: basic properties of D(X) and M(X)} (P3)), so \eqref{temp1} becomes
\begin{equation}\label{temp2}
\mathrm{Hom}_{\mathcal{D}(M)}\left(\mathrm{res}^G_MA,A_1\right)\cong\mathrm{Hom}_{\mathcal{D}(M\times G)}\left(\widetilde{\beta}Rf_!(\pi''_{M,G})^*A[\dim U^{\pm}_{G-M}],\widetilde{\beta}A'_1\right).
\end{equation}
And, the condition $\mathrm{res}^G_MA\in\mathcal{D}^{\leq0}(M)$ becomes (Remark~\ref{remark: basic properties of D(X) and M(X)} (P1))
\begin{equation*}
\widetilde{\beta}Rf_!(\pi''_{M,G})^*A[\dim U^{\pm}_{G-M}]\in\mathcal{D}^{\leq0}(M\times G),
\end{equation*}
thus Remark~\ref{remark: basic properties of D(X) and M(X)} (P1) and the adjunctions imply that
\begin{equation*}
\begin{split}
\eqref{temp2}
=&\mathrm{Hom}_{\mathcal{D}(MU^{\pm}\backslash M\times G)}\left(Rf_!(\pi''_{M,G})^*A[\dim U^{\pm}_{G-M}],A'_1\right)\\
=&\mathrm{Hom}_{\mathcal{D}(G)}\left(A,R(\pi''_{M,G})_*f^!A'_1[-\dim U^{\pm}_{G-M}]\right).
\end{split}
\end{equation*}

Now we want to apply the Verdier duality. First note that, if $K$ is an equivariant perverse sheaf on $M$, then by the compatibilities between the Verdier duality functor and (proper) pull-backs (see e.g.\ \cite[E.4]{Boyarchenko_Drinfeld2010CharSh}), its Verdier dual $\mathbb{D}_MK$ is also equivariant. We have
\begin{equation}\label{temp3}
\begin{split}
R(\pi''_{M,G})_*f^!A'_1[-\dim U^{\pm}_{G-M}]
&=R(\pi''_{M,G})_*f^!(\mathbb{D}_{MU^{\pm}\backslash M\times G}\circ\mathbb{D}_{MU^{\pm}\backslash M\times G})A'_1[-\dim U^{\pm}_{G-M}]\\
&=\left(\mathbb{D}_G\circ R(\pi''_{M,G})_!f^*\mathbb{D}_{MU^{\pm}\backslash M\times G}A'_1\right)[-\dim U^{\pm}_{G-M}]\\
&=\mathbb{D}_G\left(R(\pi''_{M,G})_!f^*(\mathbb{D}_{MU^{\pm}\backslash M\times G}A'_1)[\dim U^{\pm}_{G-M}]\right)\\
&=\mathbb{D}_GR(\pi''_{M,G})_!\widetilde{f}\mathbb{D}_{MU^{\pm}\backslash M\times G}A'_1
\end{split}
\end{equation}
by the Verdier duality (see e.g.\ \cite[E.4]{Boyarchenko_Drinfeld2010CharSh}). 

\vspace{2mm} Meanwhile, since $\widetilde{\gamma}A_1=\widetilde{\beta}A'_1$, we see that
\begin{equation*}
\begin{split}
\widetilde{\gamma}\mathbb{D}_MA_1
&={\gamma}^*(\mathbb{D}_MA_1)[\dim G]=\mathbb{D}_{M\times G}(\gamma^!A_1[-\dim G])=\mathbb{D}_{M\times G}\widetilde{\gamma}A_1\\
&=\mathbb{D}_{M\times G}\widetilde{\beta}A'_1=\beta^!(\mathbb{D}_{MU^{\pm}\backslash M\times G}A'_1)[-\dim MU^{\pm}]=\widetilde{\beta}\mathbb{D}_{MU^{\pm}\backslash M\times G}A'_1,
\end{split}
\end{equation*}
so
\begin{equation*}
\widetilde{\rho}(\mathbb{D}_MA_1)_{M,G}\cong\widetilde{\phi}\widetilde{\gamma}\mathbb{D}_MA_1\cong\widetilde{\phi}\widetilde{\beta}\mathbb{D}_{MU^{\pm}\backslash M\times G}A'_1\cong\widetilde{\rho}\widetilde{f}\mathbb{D}_{MU^{\pm}\backslash M\times G}A'_1.
\end{equation*}
Thus by the uniqueness of $(\mathbb{D}_MA_1)_{M,G}$ (see \eqref{define (-)_{M,G}}) we get $\widetilde{f}\mathbb{D}_{MU^{\pm}\backslash M\times G}A'_1\cong(\mathbb{D}_MA_1)_{M,G}$. Therefore 
\begin{equation*}
\eqref{temp3}=\mathbb{D}_GR(\pi''_{M,G})_!(\mathbb{D}_MA_1)_{M,G}=\mathbb{D}_G\circ\mathrm{ind}_M^G\circ\mathbb{D}_MA_1.
\end{equation*}
This completes the proof.
\end{proof}

\section{Further remarks on function fields}\label{further remarks}

Throughout this section we focus on the function field case $\mathcal{O}={\mathbb{F}}_q[[\pi]]$. Note that a very special phenomenon happened in this case is the existence of a natural section $G_1\rightarrow G$ to the reduction map.

\vspace{2mm} In this situation, Lusztig proposed in \cite{Lusz_CharSh_Generalizations} the generic principal series character sheaves (here ``principal series'' means that the involved $T$ is contained in an $F$-stable $B$), and conjectured that these complexes are perverse; this construction is based on an analogue of the Grothendieck--Springer resolution. In \cite{Lusztig_2015_generic_CharSh}, Lusztig proved the conjecture for $r\leq4$, with $\mathrm{char}(\mathbb{F}_q)$ big enough; in the arguments, some complexes based on G\'erardin's representations are proved to be perverse (the construction of these complexes is a little bit different from the one presented here, but the resulting complexes are up to shifts isomorphic), and he showed that the principal series character sheaves coincide with these complexes (see also \cite{Kim_2016_PincipalSeriesComparison}, which extended this coincidence for a general $r$).

\vspace{2mm} Meanwhile, at level $r=2$, Fan considered in his PhD thesis \cite{Fan_Zhaobing_PhD_thesis} another construction of character sheaves, by a method different from ours. While in his construction the characteristic functions are not known, it is immediate from the definition that his complexes are perverse. It would be interesting to understand the relations between his construction and our construction.

\bibliographystyle{alpha}
\bibliography{zchenrefs}

\begin{thebibliography}{BBD82}

\bibitem[BBD82]{BBD}
Alexander Beilinson, Joseph Bernstein, and Pierre Deligne.
\newblock Faisceaux pervers.
\newblock In {\em Analysis and topology on singular spaces, {I} ({L}uminy,
  1981)}, volume 100 of {\em Ast\'erisque}, pages 5--171. Soc. Math. France,
  Paris, 1982.

\bibitem[BD10]{Boyarchenko_Drinfeld2010CharSh}
Mitya Boyarchenko and Vladimir Drinfeld.
\newblock A motivated introduction to character sheaves and the orbit method
  for unipotent groups in positive characteristic.
\newblock {\em arXiv preprint arXiv:0609769v2}, 2010.

\bibitem[Bor91]{Borel_1991_LinearAlgGp}
Armand Borel.
\newblock {\em Linear algebraic groups}, volume 126 of {\em Graduate Texts in
  Mathematics}.
\newblock Springer-Verlag, New York, second edition, 1991.

\bibitem[Che17]{ZheChen_PhDthesis}
Zhe Chen.
\newblock On generalised {D}eligne--{L}usztig constructions.
\newblock 2017.
\newblock {P}h{D} {T}hesis. {D}urham {U}niversity.

\bibitem[CS17]{ChenStasinski_2016_algebraisation}
Zhe Chen and Alexander Stasinski.
\newblock The algebraisation of higher {D}eligne--{L}usztig representations.
\newblock {\em Selecta Math. (N.S.)}, 23(4):2907--2926, 2017.

\bibitem[Del77]{SGA412}
Pierre Deligne.
\newblock {\em Cohomologie \'etale}.
\newblock Springer, 1977.
\newblock S{\'e}minaire de G{\'e}om{\'e}trie Alg{\'e}brique du Bois-Marie (SGA
  4$\frac{1}{2}$).

\bibitem[Del80]{deligne_conjecture_1980}
Pierre Deligne.
\newblock La conjecture de {W}eil. {II}.
\newblock {\em Inst. Hautes \'Etudes Sci. Publ. Math.}, (52):137--252, 1980.

\bibitem[DG70]{SGA3}
Michel Demazure and Alexander Grothendieck.
\newblock {\em {S}ch{\'e}mas en {G}roupes. S{\'e}minaire de {G}{\'e}om{\'e}trie
  {A}lg{\'e}brique du {B}ois {M}arie 1962/64 ({SGA} 3)}.
\newblock Springer-Verlag, 1970.

\bibitem[DL76]{DL1976}
Pierre Deligne and George Lusztig.
\newblock Representations of reductive groups over finite fields.
\newblock {\em Ann. of Math. (2)}, 103(1):103--161, 1976.

\bibitem[Fan12]{Fan_Zhaobing_PhD_thesis}
Zhaobing Fan.
\newblock {\em Geometric approach to {H}all algebras and character sheaves}.
\newblock ProQuest LLC, Ann Arbor, MI, 2012.
\newblock Thesis (Ph.D.)--Kansas State University.

\bibitem[GD67]{grothendieck_elements_1967}
Alexander Grothendieck and Jean Dieudonn\'e.
\newblock {\em \'{E}l\'ements de g\'eom\'etrie alg\'ebrique {IV} \'{E}tude
  locale des sch\'emas et des morphismes de sch\'emas}.
\newblock Inst. Hautes \'Etudes Sci. Publ. Math., 1964-1967.

\bibitem[G{\'e}r75]{Gerardin1975SeriesDiscretes}
Paul G{\'e}rardin.
\newblock {\em Construction de s\'eries discr\`etes {$p$}-adiques}.
\newblock Lecture Notes in Mathematics, Vol. 462. Springer-Verlag, Berlin-New
  York, 1975.

\bibitem[Gre61]{Greenberg19611}
Marvin~J. Greenberg.
\newblock Schemata over local rings.
\newblock {\em Ann. of Math. (2)}, 73:624--648, 1961.

\bibitem[Gre63]{Greenberg19632}
Marvin~J. Greenberg.
\newblock Schemata over local rings. {II}.
\newblock {\em Ann. of Math. (2)}, 78:256--266, 1963.

\bibitem[Kim16]{Kim_2016_PincipalSeriesComparison}
Dongkwan Kim.
\newblock A comparison of two complexes.
\newblock {\em arXiv preprint arXiv:1603.03845}, 2016.

\bibitem[Lau87]{Laumon_1987_Transformation_Fourier}
G\'erard Laumon.
\newblock Transformation de {F}ourier, constantes d'\'equations fonctionnelles
  et conjecture de {W}eil.
\newblock {\em Inst. Hautes \'Etudes Sci. Publ. Math.}, (65):131--210, 1987.

\bibitem[Let05]{Let2005book}
Emmanuel Letellier.
\newblock {\em Fourier transforms of invariant functions on finite reductive
  {L}ie algebras}, volume 1859 of {\em Lecture Notes in Mathematics}.
\newblock Springer-Verlag, Berlin, 2005.

\bibitem[Lus79]{Lusztig1979SomeRemarks}
George Lusztig.
\newblock Some remarks on the supercuspidal representations of {$p$}-adic
  semisimple groups.
\newblock In {\em Automorphic forms, representations and {$L$}-functions,
  {P}art 1}, pages 171--175. Amer. Math. Soc., 1979.

\bibitem[Lus85]{Lusztig_CharSh_I}
George Lusztig.
\newblock Character sheaves {I}.
\newblock {\em Adv. in Math.}, 56(3):193 -- 237, 1985.

\bibitem[Lus04]{Lusztig2004RepsFinRings}
George Lusztig.
\newblock Representations of reductive groups over finite rings.
\newblock {\em Represent. Theory}, 8:1--14, 2004.

\bibitem[Lus06]{Lusz_CharSh_Generalizations}
George Lusztig.
\newblock Character sheaves and generalizations.
\newblock In {\em The unity of mathematics}, volume 244 of {\em Progr. Math.},
  pages 443--455. Birkh\"auser Boston, Boston, MA, 2006.

\bibitem[Lus10]{Lusztig_1994_IntroQuantum}
George Lusztig.
\newblock {\em Introduction to quantum groups}.
\newblock Modern Birkh\"auser Classics. Birkh\"auser/Springer, New York, 2010.
\newblock Reprint of the 1994 edition.

\bibitem[Lus17]{Lusztig_2015_generic_CharSh}
George Lusztig.
\newblock Generic character sheaves on groups over {${\mathbf
  k}[\epsilon]/(\epsilon^r)$}.
\newblock In {\em Categorification and higher representation theory}, volume
  683 of {\em Contemp. Math.}, pages 227--246. Amer. Math. Soc., Providence,
  RI, 2017.

\bibitem[Ser58]{serre_1958_special_gp}
Jean-Pierre Serre.
\newblock Espaces fibr{\'e}s alg{\'e}briques.
\newblock {\em S{\'e}minaire Claude Chevalley}, 3:1--37, 1958.

\bibitem[Sta09]{Sta2009Unramified}
Alexander Stasinski.
\newblock Unramified representations of reductive groups over finite rings.
\newblock {\em Represent. Theory}, 13:636--656, 2009.

\bibitem[Sta12]{Sta2012ReductiveGr}
Alexander Stasinski.
\newblock Reductive group schemes, the {G}reenberg functor, and associated
  algebraic groups.
\newblock {\em J. Pure Appl. Algebra}, 216(5):1092--1101, 2012.

\end{thebibliography}

\end{document}